\documentclass[10pt]{article}
\usepackage{graphicx,bm,amssymb,amsmath,amsthm}
\usepackage{afterpage}

\newcommand{\bi}{\begin{itemize}}
\newcommand{\ei}{\end{itemize}}
\newcommand{\ben}{\begin{enumerate}}
\newcommand{\een}{\end{enumerate}}
\newcommand{\be}{\begin{equation}}
\newcommand{\ee}{\end{equation}}
\newcommand{\bea}{\begin{eqnarray}} 
\newcommand{\eea}{\end{eqnarray}}
\newcommand{\ba}{\begin{align}} 
\newcommand{\ea}{\end{align}}
\newcommand{\bse}{\begin{subequations}} 
\newcommand{\ese}{\end{subequations}}
\newcommand{\bc}{\begin{center}}
\newcommand{\ec}{\end{center}}
\newcommand{\bfi}{\begin{figure}}
\newcommand{\efi}{\end{figure}}
\newcommand{\ca}[2]{\caption{#1 \label{#2}}}
\newcommand{\ig}[2]{\includegraphics[#1]{#2}}
\newcommand{\bmp}[1]{\begin{minipage}{#1}}
\newcommand{\emp}{\end{minipage}}
\newcommand{\pig}[2]{\bmp{#1}\includegraphics[width=#1]{#2}\emp}
\newcommand{\tbox}[1]{{\mbox{\tiny #1}}}
\newcommand{\mbf}[1]{{\mathbf #1}}
\newcommand{\half}{\mbox{\small $\frac{1}{2}$}}
\newcommand{\RR}{\mathbb{R}}
\newcommand{\ZZ}{\mathbb{Z}}
\newcommand{\eps}{\varepsilon}
\newcommand{\bigO}{{\mathcal O}}

\DeclareMathOperator{\re}{Re}
\DeclareMathOperator{\im}{Im}
\DeclareMathOperator{\sinc}{sinc}
\newtheorem{thm}{Theorem}
\newtheorem{cnj}[thm]{Conjecture}
\newtheorem{lem}[thm]{Lemma}

\newtheorem{rmk}[thm]{Remark}
\newcommand{\freq}{\beta}          %
\newcommand{\rat}{\sigma}          %
\newcommand{\ppsi}{{\tilde\psi}}   %
\newcommand{\al}{\alpha}           %
\newcommand{\NU}{{nonuniform}}       %

\newcommand{\KB}{Kaiser--Bessel}
\newcommand{\FT}{Fourier transform}

\begin{document}

\title{Aliasing error of the exp$(\beta \sqrt{1-z^2})$ kernel
  in the nonuniform fast Fourier transform}

\author{Alex H. Barnett%
  \thanks{Center for Computational Mathematics, Flatiron Institute, Simons Foundation, New York, NY, USA}
}
\date{\today}
\maketitle
\begin{abstract}
  The most popular algorithm for the
   \NU\ fast Fourier transform 
   (NUFFT) uses the dilation of a kernel $\phi$ to spread (or interpolate)
   between given \NU\ points and a uniform upsampled grid,
  combined with an FFT and diagonal scaling
  (deconvolution) in frequency space.
  The high performance of the recent FINUFFT library
  is in part due to its use of
  a new ``exponential of semicircle'' kernel
  $\phi(z)=e^{\beta \sqrt{1-z^2}}$, for $z\in[-1,1]$, zero otherwise, whose
  Fourier transform $\hat\phi$ is unknown analytically.
  We place this kernel on a rigorous footing
  by proving an aliasing error estimate
  which 
  bounds the error of the
  one-dimensional NUFFT of types 1 and 2
  in exact arithmetic.
  Asymptotically in the kernel width measured in upsampled grid points,
  the error is shown to decrease with an exponential rate
  arbitrarily close to
  that of the popular Kaiser--Bessel kernel.
  This requires controlling a
  conditionally-convergent
  sum over the tails of $\hat\phi$,
  using steepest descent, other classical estimates on contour integrals,
  and a phased sinc sum.
  We also draw new connections between the above kernel,
  Kaiser--Bessel, and prolate spheroidal wavefunctions of order zero,
  which all appear to share an optimal exponential convergence rate.
\end{abstract}

\section{Introduction and main result} %
\label{s:intro}

The NUFFT computes exponential sums involving
arbitrary off-grid source or target points,
at speeds scaling like those of the FFT for regular grids.
It has a wide range of applications, including
magnetic resonance imaging \cite{jackson91,fessler,pynufft},
computed tomography \cite{fourmont},
optical coherence tomography \cite{octnufft},
synthetic aperture radar \cite{andersson12},
spectral interpolation between grids \cite[Sec.~6]{usingnfft} \cite{gimbutasgrid}, and electrostatics in molecular dynamics \cite{nestlerPME,shamshirgar};
for reviews see \cite{usingnfft,nufft,finufft}.
Given \NU\ points $x_j$, $j=1,\ldots,M$,
which may be taken to lie in $[-\pi,\pi)$,
complex strengths $c_j$, and a bandwidth $N\in2\mathbb{N}$,
the 1D type 1 NUFFT computes the $N$ outputs
\be
f_k :=
\sum_{j=1}^M c_j e^{i k x_j}~,
\quad %
-N/2\le k < N/2 ~.
\label{1}
\ee
The type 2 is the adjoint of this operation:
it computes at arbitrary real targets $x_j$ the $N$-term Fourier series with given
coefficients $f_k$,
  \be
  c_j := \sum_{-N/2\le k < N/2} f_k e^{-i k x_j}~,
  \quad j=1,\dots, M~.
\label{2}
\ee
The above naturally generalize to dimension $d>1$, and there are
related flavors of forward and inverse tasks that we will not address here
\cite{dutt,usingnfft}.
Naively the sums \eqref{1} or \eqref{2} require $\bigO(NM)$ work
($N$ being the total number of modes in cases with $d>1$);
NUFFT algorithms approximate them
to a user-specified relative tolerance $\eps$
with typically only $\bigO\bigl(M(\log 1/\eps)^d + N \log N\bigr)$ work.

The most popular algorithm (see, e.g., \cite{fessler,usingnfft,finufft}) for the type 1,
taking the 1D case, fixes a fine grid with
nodes $2\pi l /n$, for $l=0,\dots,n-1$, where $n = \sigma N$, and $\sigma>1$
is an upsampling parameter.
The data is first spread to a vector $\{b_l\}_{l=0}^{n-1}$ living on this grid via
\be
b_l = \sum_{j=1}^M c_j \tilde\psi(2\pi l/n - x_j)~, \qquad %
l=0,\dots,n-1~,
\label{bl}
\ee
where $\tilde\psi(x) := \sum_{m\in\ZZ} \psi(x-2\pi m)$ is the periodization
of a ``scaled'' (dilated) %
kernel $\psi(x) := \phi(nx/\pi w)$,
where $\phi$ is some ``unscaled'' kernel, meaning its support is $[-1,1]$.
Thus $w\in\mathbb{N}$ is the scaled kernel width in fine grid points.
The output approximations $\tilde f_k \approx f_k$ to \eqref{1} are then
\be
\tilde f_k = p_k \sum_{l=0}^{n-1} e^{2\pi i lk/n} b_l ~, \qquad %
-N/2\le k < N/2~,
\label{pb}
\ee
which is computed via a size-$n$ FFT followed by truncation to the
desired output indices.
The total effort is thus $\bigO(M w + \sigma N \log N)$,
or $\bigO(M w^d + \sigma^d N \log N)$ for $d>1$.
The ``deconvolution'' factors $p_k$ in \eqref{pb} are designed to undo
the kernel convolution in \eqref{bl}, thus are usually
chosen as
\be
p_k = \frac{2\pi}{n\hat\psi(k)} = \frac{2}{w\hat\phi(\pi w k/n)}~,
\qquad -N/2\le k < N/2
\label{pk}
~,
\ee
where $\hat{}$\, indicates the Fourier transform according to the definition
$\hat\psi(k) = \int_{-\infty}^\infty \psi(x) e^{ikx} dx$.
The type 2 performs the adjoints of the above steps in reverse order, to give
approximations $\tilde c_j \approx c_j$ to \eqref{2}.
The goal of this paper is, given $\sigma$, $w$, and a particular kernel $\phi(z)$, to bound the
output errors $\tilde f_k - f_k$ or $\tilde c_j-c_j$ of this algorithm
in exact arithmetic.
These errors are due to {\em aliasing}, that is, the size of $\hat\psi(k)$
in the ``tails'' $|k|\ge n-N/2$ relative to its %
size in the output band $|k|\le N/2$.

In the earliest NUFFT algorithms with rigorous error analysis,
$\phi$ was a truncated Gaussian \cite{dutt,steidl98}
or B-spline \cite{beylkinnufft},
and exponential (geometric) convergence in $w$ was shown,
i.e.\ the error is $\eps = \bigO(e^{-c_\sigma w})$.
At the standard choice $\sigma=2$,
the rates $c_\sigma$ for both of these kernels correspond to
around $0.5 w$ digits of accuracy \cite{nufft}.
However, because of the spreading cost $Mw^d$, minimizing $w$ for a given
$\sigma$ and $\eps$ is crucial.
It was thus a (perhaps surprising) discovery of Logan and Kaiser \cite{kaiser,kaiserinterview,jackson91} that the ``\KB'' (KB) Fourier transform pair
\cite{dftsubmat}
(plotted in green in Fig.~\eqref{f:kernel})
\be
\phi_{\tbox{KB},\freq}(z) := \left\{
\begin{array}{ll}\frac{I_0(\freq\sqrt{1-z^2})}{I_0(\freq)}, & |z|\le 1\\
  0~,& \mbox{otherwise,}\end{array}\right.
\qquad
\hat\phi_{\tbox{KB},\freq}(\xi) = \frac{2}{I_0(\freq)}
\frac{\sinh \sqrt{\freq^2-\xi^2}}{\sqrt{\freq^2-\xi^2}}
~, \quad \xi\in\RR~,
\label{KBpair}
\ee
where $I_0$ is the modified Bessel function \cite[(10.25.2)]{dlmf},
has a much higher rate, achieving over
$0.9 w$ digits of accuracy (for an optimal scaling of parameter $\beta$ with $w$)
at $\sigma=2$; see \eqref{KBerr}.
The rigorous analysis, due to Fourmont~\cite{fourmontthesis,fourmont}, is subtle.
This rate is the conjectured best possible
(see Sec.~\ref{s:optim}),
namely that of the {\em prolate spheroidal wavefunction} (PSWF) of order zero
\cite{SlepianI,osipov}, the latter being the
minimizer of the tail mass $\|\hat\phi\|_{L^2(\{|\xi|>\freq\})}$ over functions
$\phi$ with support $[-1,1]$.
Thus (and because it is simpler to evaluate that the PSWF),
KB is popular in NUFFT code libraries,
either in ``forward mode'' (spreading with $\phi_{\tbox{KB},\freq}$)
\cite{MIRT,BART,pynufft},
or ``backward mode'' (spreading with a truncated $\hat\phi_{\tbox{KB},\freq}$)
\cite{usingnfft}.

Recently the author and coworkers released a library \cite{finufft}
whose high speed is in part due to the use of a new ``exponential
of semicircle'' (ES) kernel,
\be
\phi_{\tbox{ES},\freq}(z) :=
\left\{\begin{array}{ll}
e^{\freq (\sqrt{1-z^2}-1)}, & |z|\le 1~,\\
0, & \mbox{otherwise}~,
\end{array}
\right.
\label{ES}
\ee
being simpler than either KB or PSWF, yet
empirically having the same optimal rate and very similar errors.
The ES kernel has since been used to accelerate Spectral Ewald codes
for periodic electrostatic sums \cite{shamshirgar},
and fast gridding in radio interferometry \cite{arras20}.
Its Fourier transform $\hat\phi_{\tbox{ES},\freq}(\xi)$ is not known analytically,
yet is easily evaluated by quadrature \cite[Sec.~3.1.1]{finufft}.
The above three kernels are compared in Fig.~\ref{f:kernel} for
small and large $\beta$ parameters.

\bfi[t!]  %
\ig{width=\textwidth}{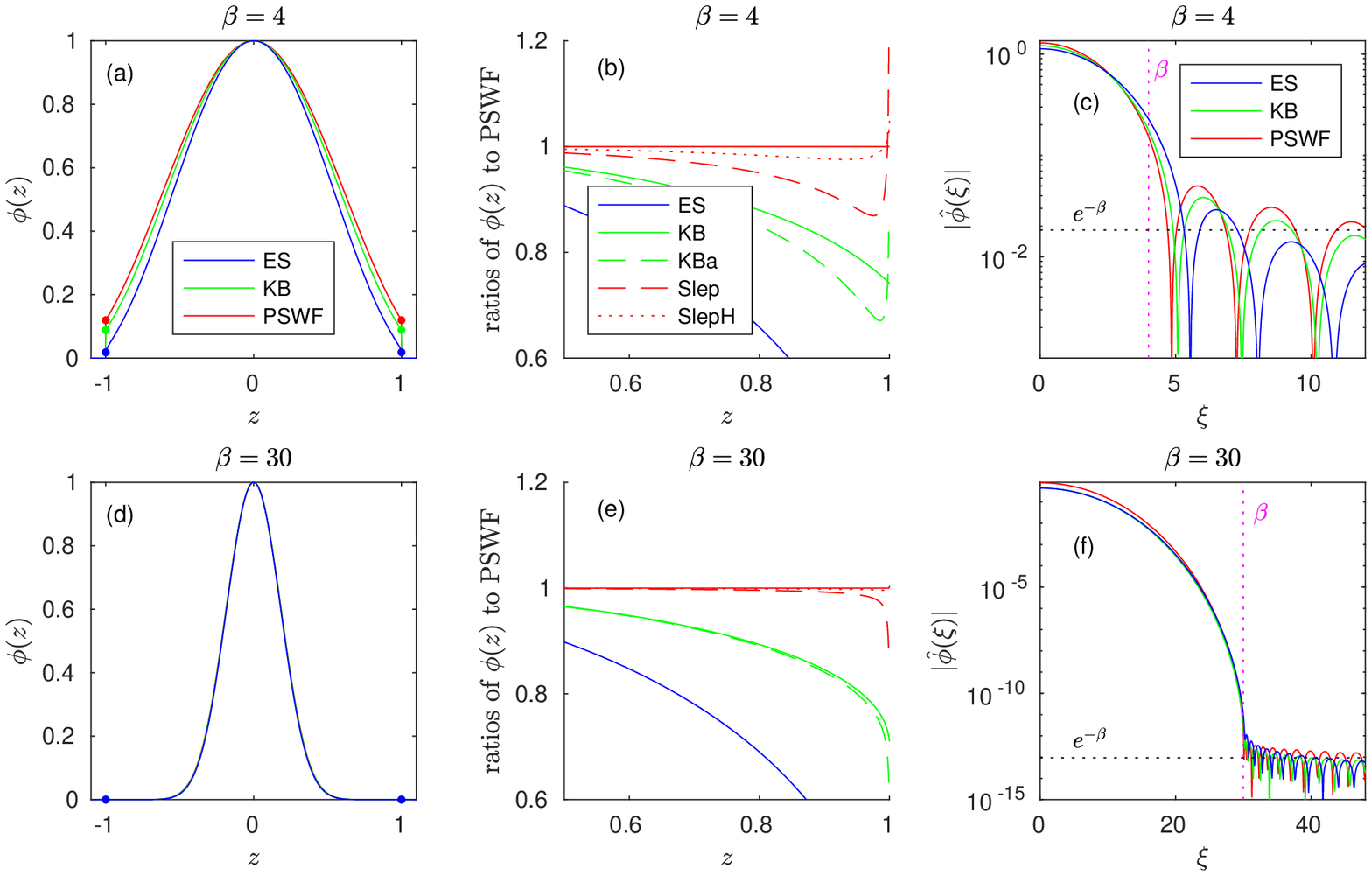}
\ca{Comparison of the three unscaled spreading kernels $\phi(z)$ on $[-1,1]$:
  exponential of semicircle \eqref{ES} (ES, blue), \KB\ \eqref{KBpair}
  (KB, green),
  and the PSWF of order zero (blue; see Sec.~\ref{s:optim}).
  (a) The three kernels for parameter $\freq=4$.
  Discontinuities at $\pm 1$ are shown by dots.
  (b) Ratios to the PSWF, i.e.\ $\phi(z)/\psi_0(z)$, for:
the other two kernels (solid lines), and asymptotic approximations
  ``KBa'' \eqref{KBasym} (dashed green),
  ``Slep'' (Slepian's upper form in \eqref{Slep}, dashed red),
  and ``SlepH'' (a hybrid form \eqref{SlepH}, dotted red).
  (c) Magnitude of the three kernel Fourier transforms.
  All three have logarithm close to a quarter-ellipse below
  the cutoff frequency $\xi=\freq$ (vertical dotted line), and
  are roughly bounded by $e^{-\freq}$ (horizontal dotted line)
  above cutoff.
  The bottom row (d--f) shows the same as (a--c) but for $\freq=30$.
  In (d) the three kernel graphs are indistinguishable.
}{f:kernel}
\efi    %

The main goal of this paper is to prove the following
aliasing error convergence theorem applying to \eqref{ES},
which places this simple kernel
(and hence algorithms which use it \cite{finufft,shamshirgar})
on a rigorous footing.
The rate will depend on the fixed upsampling factor $\rat>1$.
We also assume, as is usual with KB \cite{fessler},
a kernel parameter $\freq$ proportional to $w$. Specifically,
\be
\freq(\rat,\gamma,w) \;:=\; \gamma \pi w (1-1/2\rat)~,
\label{gam}
\ee
where $\gamma$ is a ``safety factor''.
At $\gamma=1$, the cutoff
(see Fig.~\ref{f:kernel}(c,f))
in $\hat\psi_{\tbox{ES},\beta}$ would coincide with the lowest aliased frequency $k=n-N/2$,
so in practice one sets $\gamma<1$
(see Remark~\ref{r:safety}).
Our main result concerns the constant $\eps_\infty$ in the standard
$\ell^1$-$\ell^\infty$ output error bounds
\be
\max_{-N/2\le k < N/2}|\tilde f_k - f_k| \le \eps_\infty \|\mbf{c}\|_1
\quad \mbox{ (type 1), } \qquad
\max_{1\le j\le M}|\tilde c_j - c_j| \le \eps_\infty \|\mbf{f}\|_1
\quad \mbox{ (type 2), }
\label{1nrm}
\ee
where we use vector notation $\mbf{c}:=\{c_j\}_{j=1}^M$ and $\mbf{f}:=\{f_k\}_{k=-N/2}^{N/2-1}$.

\begin{thm}(stated without proof as \cite[Thm.~7]{finufft}).
  Fix the number of modes $N\in\mathbb{N}$, the upsampled grid size $n>N$
  (hence the upsampling factor $\rat=n/N>1$),
  and the safety factor $\gamma\in(0,1)$.
  Then the constant bounding the
  error \eqref{1nrm} for the 1D type 1 and 2 NUFFT in exact arithmetic,
  using the ES kernel \eqref{ES} with $\freq=\freq(\rat,\gamma,w)$
  defined by \eqref{gam},
  converges with respect to the kernel width $w$ as
  \be
  \eps_\infty \;=\; \bigO\left( \sqrt{w} e^{-\pi w \gamma
    \sqrt{1-1/\rat - (\gamma^{-2}-1)/4\rat^2}}
  \right)
  ~, \qquad w\to \infty ~.
  \label{ESerr}
  \ee
  \label{t:ESerr}
\end{thm}  %
The rest of this paper breaks its proof into three stages:
Section~\ref{s:err} reviews the standard bound for
$\eps_\infty$ involving a phased sum over the
Fourier transform $\hat\psi$ of a general scaled kernel.
Section~\ref{s:asymp} uses contour deformation and steepest descent
to derive $\freq\to\infty$ asymptotics
for %
$\hat\phi_{\tbox{ES},\freq}$,
both below and above cutoff,
then proves two technical lemmas on the decay of $\hat\phi_{\tbox{ES},\freq}$.
Section~\ref{s:pf} brings in two lemmas to handle
phased sinc sums, then combines all of these ingredients to complete the proof.

We conclude the paper in Section~\ref{s:optim} by
drawing new connections---apparent in Fig.~\ref{f:kernel}(b,e)---between the
ES, KB, and PSWF kernels, and discussing their shared optimal rate.

\begin{rmk}[Comparison to \KB\ bounds]  %
  \label{r:fourmont}
  In the limit $\gamma\to1^{-}$, \eqref{ESerr}
  approaches the exponential convergence rate
  of the rigorous estimate
  for the KB kernel
  \cite{fourmont}
  \cite[App.~C]{usingnfft}
  when its parameter $\freq$ is set by \eqref{gam} with $\gamma=1$,
  \be
  \eps_\infty \;\le\; 4\pi   (1-1/\rat)^{1/4}
  \left(\sqrt{\frac{w-1}{2}}+\frac{w-1}{2}\right)
  e^{-\pi(w-1)\sqrt{1-1/\rat}}
  ~.
  \label{KBerr}
  \ee
  Unlike our result, this is a non-asymptotic bound with explicit constant,
  although we note that our algebraic prefactor is
  improved by a factor $\sqrt{w}$.
\end{rmk}   %

\begin{rmk}[Choice of safety factor $\gamma$]  %
  Taking $\gamma\to 1^-$ maximizes the exponential
  rate in \eqref{ESerr},
  but, in practice, choosing it slightly below 1 gives the smallest error:
  in \cite{finufft} we recommend $\gamma \approx 0.98$, similar to
  previous workers \cite[Table~II]{jackson91} \cite[Fig.~11]{fessler}.
  The restriction $\gamma<1$ in Thm.~\ref{t:ESerr}
  is due to
  breakdown of the stationary phase estimate \eqref{EShat2}
  at the cutoff frequency $\xi=\freq$ where saddles head to $\pm\infty$.
  One may be able to extend the proof to $\gamma=1$ by using
  another method to bound $\hat\phi$ near cutoff.
  However, because of the rapid growth in $\hat\phi$ below cutoff
  (see, e.g., Fig.~\ref{f:kernel}(e)),
  for any $\gamma>1$ the rate would necessarily be severely reduced.
  \label{r:safety}
\end{rmk}        %

\begin{rmk}[Related work] %
  After we submitted this work,
  the 2nd version of a recent preprint by Potts--Tasche
  appeared, including a remarkable theorem \cite[Thm.~4.5]{potts19}
  showing the same exponential convergence
  rate in $\eps_\infty$ as we do,
  but for the related kernel
\be
\phi_{\tbox{\rm sinh},\freq}(z) :=
\left\{\begin{array}{ll}
\sinh \freq \sqrt{1-z^2}, & |z|\le 1~,\\
0, & \mbox{otherwise}~,
\end{array}
\right.
\label{sinh}
\ee
and only for the parameter scaling $\freq=2w$.
Their restriction
$\rat>1/[2(1-2/\pi)] \approx 1.376$ corresponds to $\gamma<1$
for their choice $\freq=2w$.
This excludes some low-upsampling options
such as $\rat = 5/4$ that are useful in practice \cite{finufft}.
Yet, by exploiting the analytic
Fourier transform \cite[(7.58)]{oberhettinger} of \eqref{sinh},
they obtain an {\em explicit, non-asymptotic} bound with
tighter algebraic prefactor than ours.
Thus it is now possible that better bounds than Theorem~\ref{t:ESerr}
could be derived by splitting \eqref{ES} into a multiple of \eqref{sinh}
plus a small correction.
\end{rmk} %

\section{Aliasing error for type 1 and type 2 transforms}%
\label{s:err}

Here we recall a known rigorous estimate on the error
of the 1D type 1 and 2 algorithms given in the introduction,
performed in exact arithmetic.
We start with the Poisson summation formula
with an extra phase $e^{i\theta}$:
for any $\psi \in L^1(\RR)$ of bounded variation
with Fourier transform $\hat\psi$,
and any lattice spacing $h>0$,
\be
\sum_{l\in\ZZ} e^{il\theta} \psi(x - lh) \; = \;
\frac{1}{h} \sum_{m\in\ZZ}
\hat\psi\biggl(-\frac{2\pi m + \theta}{h}\biggr)
\exp \biggl({i\,\frac{2\pi m + \theta}{h}x}\biggr)
~.
\label{pois}
\ee
The standard proof is that multiplication
by $e^{-ix\theta/h}$ makes the left-hand side a periodic function of $x$,
hence its Fourier series coefficients are given by the
Euler--Fourier formula.
At any points $x$ where the left-hand side is discontinuous,
one must replace $\psi(x-lh)$ by $[\psi(x^+-lh)+\psi(x^--lh)]/2$
\cite[\S 11.22]{apostol}.

We now derive the aliasing error for a general scaled kernel $\psi(x)$.
For the type 1 NUFFT, setting $h=2\pi/n$, inserting \eqref{bl} into \eqref{pb},
and subtracting from the true answer
\eqref{1} gives the error
\be
\tilde f_k - f_k  \;=\;
\sum_{j=1}^M c_j \left[ p_k \sum_{l=0}^{n-1} e^{i h l k}
  \ppsi(lh-x_j) - e^{ikx_j} \right ]
\;=:\; \sum_{j=1}^M c_j g_k(x_j)
~, \quad -N/2\le k < N/2~.
\label{t1err}
\ee
Now writing the general ordinate as $x$,
and applying Poisson summation \eqref{pois} with $\theta=hk$,
\bea
g_k(x) &:=&  p_k \sum_{l=0}^{n-1} e^{i l h k}
\ppsi(lh-x) - e^{ikx}
\;=\; p_k \sum_{l\in\ZZ} e^{ilhk} \psi(lh-x) - e^{ikx}
\nonumber \\
&&=\;
\frac{p_k}{h} \sum_{m\in\ZZ} \hat\psi(k+mn) e^{i(k+mn)x} - e^{ikx}
~.
\nonumber
\eea
The choice \eqref{pk} for $p_k$ thus
exactly kills the $m=0$ term, giving the well known
aliasing error formula
\cite{steidl98} \cite[(4.1)]{fourmont} \cite[Sec.~V.B]{fessler},
\be
g_k(x) = \frac{1}{\hat\psi(k)}\sum_{m\neq 0} \hat\psi(k+mn) e^{i(k+mn)x}
~.
\label{gkx}
\ee
Thus, since $|k|\le N/2$,
error is controlled by a phased sum over the tails of $\hat\psi$
at frequencies of magnitude at least $n-N/2$.
Since type 2 is the adjoint of type 1 (or by similar manipulations to the
above), its error is
\be
\tilde c_j - c_j \; = \! \sum_{-N/2 \le k < N/2} f_k \overline{g_k(x_j)}
~, \qquad j=1,\dots,M~.
\label{t2err}
\ee
To summarize,
let $E$ be the ``error matrix''
with elements $E_{kj} = g_k(x_j)$ given explicitly by \eqref{gkx}
and $E^\ast$ be its Hermitian adjoint,
then the output aliasing error vectors are
\be
\tilde{\mbf{f}} - \mbf{f} = E\mbf{c}
\quad \mbox{ (type 1), } \qquad\qquad
\tilde{\mbf{c}} - \mbf{c} = E^\ast\mbf{f}
\quad \mbox{ (type 2). }
\label{E}
\ee
From this the bounds \eqref{1nrm} follow immediately if we
define $\eps_\infty$ by a simple uniform bound on all matrix elements,
\be
|E_{kj}|
\;\le\;
\max_{|k|\le N/2}
\|g_k\|_\infty \;\le\;
\frac{
  \max_{|k|\le N/2, \,x\in\RR} \left|
    \sum_{m\neq 0} \hat\psi(k+mn) e^{i(k+mn)x}
  \right|
}{\min_{|k|\le N/2} |\hat\psi(k)|}
\; =: \; \eps_\infty
~.
\label{epsest}
\ee
Since the dynamic range over the output band, $\hat\psi(0)/\hat\psi(N/2)$,
is not large \cite[Remark~3]{finufft},
any lack of tightness in the second inequality in \eqref{epsest} is small.

\begin{rmk}  %
  Users of NUFFT software
  often care about relative $\ell^2$ errors, rather than absolute
  $\ell^1$-$\ell^\infty$ bounds such as \eqref{1nrm}.
  Making such bounds rigorous necessitates large prefactors
  in front of $\eps_\infty$; yet, in practice,
  relative $\ell^2$ errors match $\eps_\infty$ quite well, for reasons
  discussed in \cite[Sec.~4.2]{finufft} \cite[Sec.~4]{boettcher}.
  \label{r:emperr}
\end{rmk}  %

\section{Asymptotics of the Fourier transform of the ES kernel}
\label{s:asymp}

Here we derive asymptotics in the width parameter $\freq\to\infty$
of the \FT\ of \eqref{ES}.
From now on we abbreviate the kernel by $\phi(z)$, thus its
Fourier transform by $\hat\phi(\xi)$.
We introduce the scaled frequency
\be
\rho \;:=\; \xi/\freq ~.
\label{rho}
\ee
The cutoff $|\xi|=\freq$ (vertical line
in Fig.~\ref{f:kernel}(c,f)) is therefore at $|\rho|=1$.
The following shows that, up to weak algebraic prefactors:
(a) below cutoff $\hat\phi$ has a similar form to $\phi$ itself,
and that (b) above cutoff $\hat\phi$ is oscillatory
but uniformly exponentially small,
with the same exponential rate $e^{-\freq}$ as occurs for the KB and PSWF
kernels (see Sec.~\ref{s:optim}).

\bfi[th] %
  \pig{1.95in}{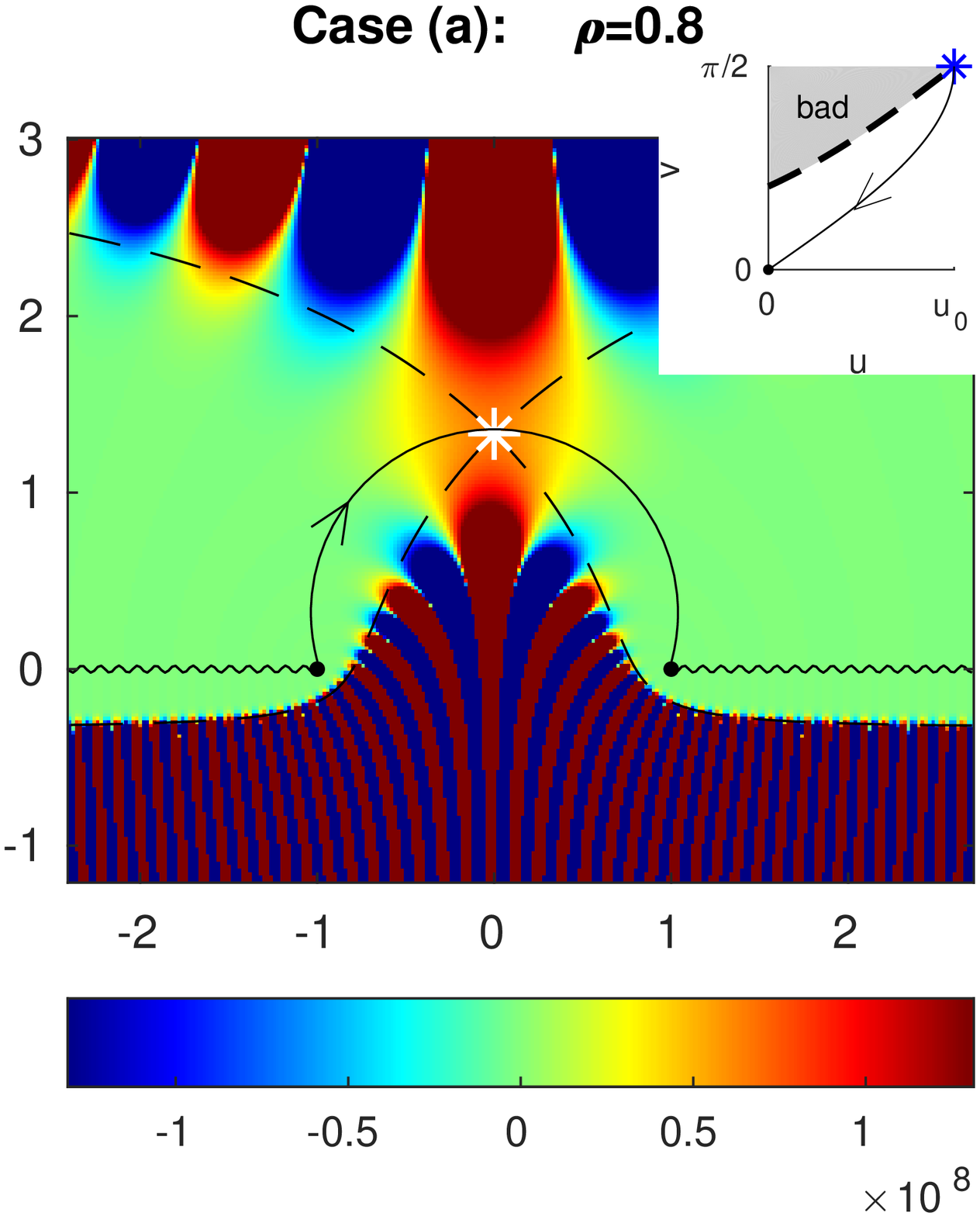}
  \pig{1.95in}{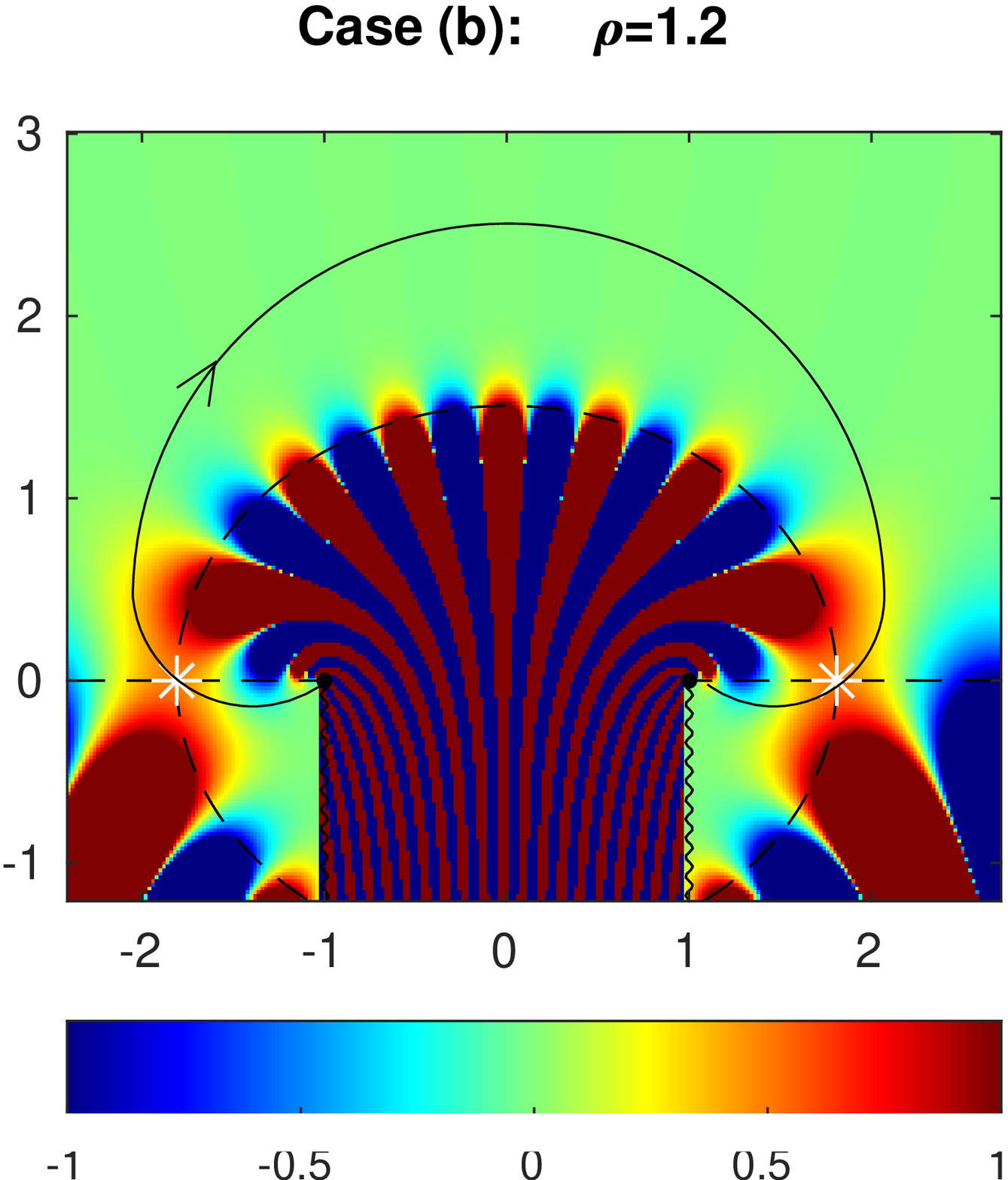}
  \;
  \pig{1.95in}{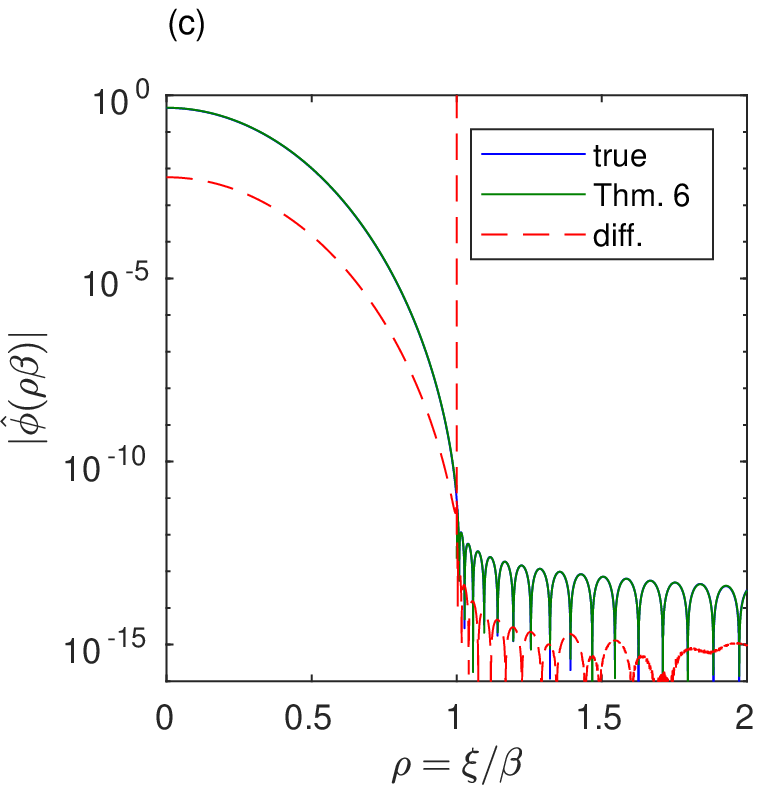}
\ca{
  Real part of the integrand $e^{\freq p(z)}$
  appearing in the ES kernel \FT\
  (see \eqref{ESFT}),
  plotted in the complex $z$ plane, for $\freq=30$.
  (a) $\rho=0.8$ (below cutoff).
  Also shown are the saddle point $z_0$ (star), example contour (curve with arrow),
  boundaries where $\re p(z) = \re p(z_0)$ (dashed lines),
  and standard branch cuts (wiggly lines).
  The inset shows the elliptic coordinate plane $(u,v)$
  for the right half of the
  $z$-plane, and the ``bad'' region where $\re p(z) > \re p(z_0)$ (shaded).
  (b) $\rho =1.2$ (above cutoff);
  note change in color scale. The branch cuts of the square-root have
  been rotated to point downwards, exposing the two saddle points.
  (c) Comparison of the asymptotics \eqref{EShat1} and \eqref{EShat2}
  to the true $\hat\phi$ (evaluated accurately by quadrature)
  for $\freq=30$; ``diff'' shows
  their absolute difference.  
  The weak algebraic divergence of \eqref{EShat1}--\eqref{EShat2} as $\rho\to1$ is highlighted by
  an %
  asymptote.
}{f:saddle}
\efi

\begin{thm} %
  Let $\hat\phi$ be the Fourier transform of the ES kernel \eqref{ES}.
  
  (a)
  Fix $\rho\in(-1,1)$, i.e.\ below cutoff. Then, %
\be
\hat\phi(\rho\freq) \; =\;
\sqrt{\frac{2\pi}{\beta}} \frac{1}{(1-\rho^2)^{3/4}} e^{\freq(\sqrt{1-\rho^2}-1)}
\left[ 1 + \bigO(\freq^{-1}) \right]
~, \qquad \freq\to\infty
~.
\label{EShat1}
\ee

(b)
  Fix $\rho$, $|\rho| > 1$, i.e.\ above cutoff. Then,
\be
\hat\phi(\rho\freq) \; =\;
2\sqrt{\frac{2\pi}{\beta}}
e^{-\freq}
\frac{\sin\left(\freq\sqrt{\rho^2-1} - \pi/4\right)}
{(\rho^2-1)^{3/4}}
\left[ 1 + \bigO(\freq^{-1}) \right]
~, \qquad \freq\to\infty
~.
\label{EShat2}
\ee
\label{t:EShat}
\end{thm} %

\begin{rmk}
  Fig.~\ref{f:saddle}(c) shows the high accuracy of these
  asymptotic formulae \eqref{EShat1}--\eqref{EShat2}
  even at the modest value $\freq=30$.
  Around two digits of {\em relative}
  accuracy are achieved everywhere shown
  except $\rho\approx 1$, where $\hat\phi$
  is already exponentially small.
  Remarkably, even details of the exponentially small tail {\em oscillations}
  at $\rho>1$ are matched to high relative accuracy.
  (However, Remark~\ref{r:decay} will show that 
  for $\rho \gg 1$ relative error in this tail must diverge.)
  \label{r:match}
\end{rmk}

\begin{proof}
In either case (a) or (b),  the \FT\ to be estimated is
\be
\hat\phi(\rho\freq) = e^{-\freq} \int_{-1}^1
e^{\freq (\sqrt{1-z^2} + i\rho z)} dz
= e^{-\freq} \int_{-1}^1 e^{\freq p(z)} dz~,\qquad
p(z):=\sqrt{1-z^2} +i\rho z~.
\label{ESFT}
\ee
We apply saddle point integration in the
complex $z$ plane (see \cite[Thm.~7.1, p.~127]{olver};
note that we have the opposite sign convention for $p(z)$).
This requires a smooth contour through an analytic region connecting $-1$
to $+1$, avoiding branch cuts,
passing through the saddle point(s), and along which $\re p(z)$
has its global maximum at the saddle point $z_0$ where $p'=0$.
The theorem then states that
$\int e^{\freq p(z)} dz = e^{\freq p(z_0)}\sqrt{\frac{2\pi}{-p''(z_0)\freq}}[1 + \bigO(\freq^{-1})]$.
Note that the standard branch cut $(-\infty,0)$ for the square-root gives
cuts for $p(z)$ at $(-\infty,-1)$ and $(1,+\infty)$;
these cuts are shown in Fig.~\ref{f:saddle}(a).

{\bf Case (a).}
We take $0\le \rho<1$, since $\hat\phi$ has even symmetry.
Since $\re i\rho z$ becomes more negative as $\im z$ grows,
the saddle $z_0 = i\rho/\sqrt{1-\rho^2}$ is on the
positive imaginary axis; see Fig.~\ref{f:saddle}(a).
To show the existence of a valid contour we switch to standard elliptical
coordinates
\be
z = \cosh(u+iv)~,\qquad \re z = \cosh u \cos v~,\qquad \im z = \sinh u \sin v
~,
\label{ellip}
\ee
where $\mu\ge 0$ and $0\le v < 2\pi$ covers the plane.
Since $1-z^2 = -\sinh^2(u+iv)$, we get
\be
\re p(z) = \re p(u,v) = (\cosh u - \rho \sinh u) \cdot \sin v
~,
\label{repz}
\ee
which shows, remarkably, that the magnitude of the exponential
in \eqref{ESFT} is %
separable in this coordinate system.
By solving $p'=0$ one finds that the saddle is at
$(u,v) = (\tanh^{-1} \rho, \pi/2)$,
where $\re p(z_0) = \sqrt{1-\rho^2}$.

Starting with the right half of the contour,
we need to show that there is some smooth open path in $(u,v)$ between
$(\tanh^{-1} \rho, \pi/2)$ and $(0,0)$ along which $\re p$ is
everywhere less than $\re p(z_0)$.
At the endpoint, $p = 0$, which is indeed less than $\re p(z_0)$;
yet it is possible that a barrier region of large $\re p$
prevents such a path from existing.
We now show that no such barrier exists.
The level curve $\re p(u,v) =\re p(z_0)$ has positive slope
in the $(u,v)$ plane (see Fig.~\ref{f:saddle}(a) inset).
This is clear since the level curve intersects each line $u=$ constant
only once in $(0,\pi/2)$, because $\sin v$ is monotonic there.
Furthermore, $\cosh u - \rho \sinh u$ is monotonically
decreasing in $[0,\tanh^{-1} \rho)$, as is apparent from its derivative,
so that the $v$-value of the intersection grows monotonically
with $u$.
This means that the ``bad'' region where $\re p(u,v) \ge \re p(z_0)$ is
confined to the upper left corner of the $(u,v)$ rectangle
(see inset), so there is no obstruction to crossing the diagonal
while remaining small.
The left half of the contour
may then be chosen as a reflection of the right about the imaginary $z$ axis.
Thus a valid saddle contour exists.

At the saddle point, $p''(z_0) = -(1-\rho^2)^{3/2}$,
so that, taking the first term in the saddle point theorem
\cite[Thm.~7.1, p.~127]{olver}
gives \eqref{EShat1}.

{\bf Case (b).}
Now $\rho>1$.
Solving $p'=0$ gives two saddle points on the real axis,
$z_0^{(\pm)} = \pm \rho/\sqrt{\rho^2-1}$, where $p(z_0^{(\pm)}) = 0$.
Since these lie on the standard branch cuts of the square-root,
to make use of saddle point integration connecting
$z=-1$ to $1$ in the upper half plane,
one must {\em rotate the branch cuts downwards} to expose more of the Riemann
sheet on which the contour lives.
With this done,
in order to avoid regions of large integrand, the contour must
first head into the lower half-plane, pass up through $z_0^{(-)}$,
into the upper half plane, down through $z_0^{(+)}$,
and finish again from the lower half-plane; see Fig.~\ref{f:saddle}(b).
Both saddles contribute equally.

To show the existence of a valid contour we examine \eqref{repz}.
The elliptical coordinates of the saddle points are $(\coth^{-1}\rho,0)$
and $(\coth^{-1}\rho,\pi)$.
The factor $\cosh u - \rho \sinh u$ vanishes on the ellipse
$u=\coth^{-1}\rho$ passing though the saddles, and,
since it is monotonically decreasing,
is positive for all smaller $u$ and negative for all larger $u$.
Thus $\re p<0$ everywhere in the upper half plane outside the ellipse
($u>\coth^{-1}\rho$, $0<v<\pi$), and in the lower half plane
inside the ellipse with the slit $[-1,1]$ omitted
($0<u<\coth^{-1}\rho$, $\pi<v<2\pi$).
Thus a smooth contour exists passing through these regions via the two saddles.

To apply the theorem one must
sum the left and right halves of the contour.
However, on each half the theorem still cannot be used directly, since
the start and end values $\re p(\pm1)=0$ are just as large
as the value at the saddles.
Thus we remove fixed pieces of the contour around $\pm1$,
allowing the theorem to be applied.
Using $p(z_0^{(\pm)}) = \pm i \sqrt{\rho^2-1}$ at the two saddles,
$|p''(z_0^{(\pm)})| = (\rho^2-1)^{3/2}$,
the steepest descent directions $e^{\mp i 3\pi/4}$,
and summing the two contributions, gives \eqref{EShat2}.

Finally, we show that the contributions due to these fixed excluded
pieces of the contour touching $\pm1$ are of lower order.
Consider an excluded piece of contour in the lower half-plane
from $1$ to $1+b$, where $b\in\mathbb{C}$, $\im b<0$,
and $1+b$ is strictly inside the ellipse described above.
We have already explained that $p<0$ on this contour, apart from at $p(1)=0$.
Writing $z=1+t$, and using $p(1+t) = i\rho + \sqrt{-2t-t^2} + i\rho t$,
The contour integral is
$$ \int_0^b e^{\freq p(1+t)} dt = e^{i\rho}\int_0^b e^{\freq P(t)} dt
~,\qquad \mbox{ where} \; P(t)= -1\sqrt{2t+t^2}+i\rho t \sim -i\sqrt{2t} \; \mbox{ for } t\to 0~.
$$
All the conditions for
Laplace's method for contour integrals \cite[Thm.~6.1, p.~125]{olver}
are met, with power $\mu=1/2$,
so its contribution is $\sim -e^{i\rho}/\freq^2 = \bigO(\freq^{-2})$,
which is $\freq^{3/2}$ times smaller than the contribution from the saddles.
The same argument applies near $-1$.
Thus these end contributions are of lower order and can be ignored in
\eqref{EShat2}.
\end{proof}

\begin{rmk}
  Informally Theorem~\ref{t:EShat} states:
  ``the \FT\ of the exponential
  of a semicircle is asymptotically the exponential of a semicircle, plus
  exponentially small tails.''
  This may be less of a surprise when it is
  recalled that the ES kernel is close to the PSWF (see Sec.~\ref{s:optim}), and
  that this ``self-Fourier-transform'' property holds exactly for the PSWF
  after truncation
  \cite{SlepianI,osipov}.
\end{rmk}

\begin{rmk}
  It is tempting to fix $\freq$ and interpret \eqref{EShat2} as a decay
  $\bigO(|\rho|^{-3/2})$, i.e.\ $\bigO(|\xi|^{-3/2})$,
  which would be directly summable when inserted into \eqref{epsest}.
However, this is false.
The kernel has discontinuities at $z=\pm 1$ of strength $e^{-\freq}$,
and is otherwise smooth, so the asymptotic must in fact be (at fixed $\freq$),
\be
\hat\phi(\xi) \;\sim\; 2 e^{-\freq} \frac{\sin \xi}{\xi} \;=\; \bigO(|\xi|^{-1})
~, \qquad |\xi|\to\infty~,
\label{sinc}
\ee
which is {\em not} absolutely convergent.
This is an order-of-limits problem:
\eqref{EShat2} cannot be applied at fixed $\freq$
in the limit $|\rho|\to\infty$, since the implied constant in the error
term is unknown and must in fact be unbounded as $|\rho|\to\infty$.
These growing saddle-point
error terms are associated with the saddles $z_0^{(\pm)}$ approaching
the square-root singularity endpoints $\pm1$.
Empirically we find that the smooth
transition from \eqref{EShat2} to \eqref{sinc} occurs around
$\xi\approx\freq^2$.
\label{r:decay}
\end{rmk}

Because of the above remark, in order to get a rigorous error estimate
we will also need the following, which
bounds the $\freq$-dependence of the {\em deviation} (denoted by $\hat D$)
from the sinc function
\eqref{sinc}, uniformly in $\freq$ and for sufficiently high frequencies $\xi$.
To avoid ambiguities involving asymptotics with two parameters
(here $\freq$ and $\xi$), for the rest of the section
we avoid ``big-$O$'' notation.

\begin{lem} %
  There exists a constant $C>0$, independent
  of the shape parameter $\freq$ and frequency $\xi$, such that for all
  $\freq \ge 2$ and $|\xi|\ge\freq^4$,
  the \FT\ of the ES kernel \eqref{ES} is
  \be
  \hat\phi(\xi) \;=\; e^{-\beta}\left[ 2\frac{\sin \xi}{\xi}
    + \hat D(\freq,\xi) \right]
  ~,\qquad \mbox{ where } \quad  |\hat D(\freq,\xi)| \; \le \;
  C \frac{\freq}{|\xi|^{5/4}}
  ~.
  \label{EShat3}
  \ee
  \label{l:EShat3}
\end{lem}
\begin{proof}
  Since $\hat\phi$ is symmetric, take $\xi>0$.
  Since the top-hat function with value $e^{-\freq}$ in $[-1,1]$, and zero
  elsewhere, has the \FT\ $2e^{-\beta} \sin \xi / \xi$,
  subtracting this top-hat from the kernel \eqref{ES} leaves the \FT\
  deviation
  $$
  \hat D(\freq,\xi) = \int_{-1}^1 (e^{\freq \sqrt{1-z^2}} - 1) e^{i\xi z} dz~.
  $$
  We deform this contour integral then apply
  Laplace-type estimates.
  Our deformed contour comprises three straight segments connecting the start
  point $-1$ to $-1+i$, from there to $1+i$, and from there to the endpoint $1$.
  We write their contributions as $\hat D(\freq,\xi) = I_1+I_2+I_3$.
  Since at all points $z$ on the middle segment we have
  $|e^{i\xi z}|\le e^{-\xi}$, and 
  $|e^{\freq \sqrt{1-z^2} +i\xi z}| \le e^{\sqrt{3} \freq - \xi}$, 
  but $\xi\ge \freq^4$,
  then $I_2$ is exponentially small as $\xi\to\infty$,
  and can be dropped.
  
  The first segment we parametrize by $z=-1+it$, then split the integral
  to give
  \bea
  I_1 &\;=\;& i e^{-i\xi} \int_{0}^1 (e^{\freq \sqrt{2it + t^2}} - 1)e^{-\xi t} dt
  \nonumber \\
  &&\;= i e^{-i\xi} \left[ \int_{0}^{\xi^{-1/2}} (e^{\freq \sqrt{t} \sqrt{2i + t}} - 1)e^{-\xi t} dt
    + \int_{\xi^{-1/2}}^1 (e^{\freq \sqrt{t} \sqrt{2i + t}} - 1)e^{-\xi t}  dt \right]
  ~.
  \nonumber
\eea
  Since
  \be
  |\sqrt{2i+t}|\le 5^{1/4} \quad  \mbox{ for } \; 0\le t\le1 ~,
  \label{5q}
  \ee
  the lower integral is bounded in magnitude by
  \be
  \left(\max_{0\le t \le \xi^{-1/2}} e^{5^{1/4}\freq \sqrt{t}} - 1 \right) \cdot
  \int_{0}^{\xi^{-1/2}} e^{-\xi t} dt~.
  \label{I1a}
  \ee
  Here since $\sqrt{t}\le \xi^{-1/4} \le 1/\freq$ we see that the first exponent
  is uniformly bounded by a constant, and using $e^x-1 \le e^c x$ for $x\le c$,
  we have that the first term in \eqref{I1a} is bounded by $C\freq/\xi^{1/4}$.
  The integral in \eqref{I1a} is bounded by $1/\xi$.
  Thus the lower integral is bounded by $C\freq/\xi^{5/4}$.

  Turning to the upper integral in $I_1$, we bound its two additive terms
  separately, 
  $$
  \int_{\xi^{-1/2}}^1 e^{-\xi t}  dt \le \frac{e^{-\sqrt{\xi}}}{\xi}~,
  \quad\mbox{and} \quad
  \int_{\xi^{-1/2}}^1 e^{\freq \sqrt{t}\, |\sqrt{2i + t}| - \xi t} dt
  \le
  \int_{\xi^{-1/2}}^1 e^{-\xi t/2}  dt \le \frac{2e^{-\sqrt{\xi}/2}}{\xi}~,
  $$
  which are both exponentially smaller than the lower integral,
  so can be dropped, giving $|I_1| \le C\freq/\xi^{5/4}$.
  Here the second integrand was bounded for all $t \ge \xi^{-1/2}$ using
  \eqref{5q} and that $2(5^{1/4}) \freq \le \xi^{3/4}$,
  which, since $\xi\ge\freq^4$, holds as long as
  $\freq^2 \ge 2(5^{1/4})$, which is satisfied if $\freq\ge 2$.
  
  The integrand on the third segment is the complex conjugate of the first,
  so $I_3 = I_1^\ast$ and $|I_3| = |I_1|$. This proves \eqref{EShat3}.
\end{proof}

The above lemma excludes $|\xi|<\beta^4$,
thus omts a growing (in $\freq$) number of terms in the sum
\eqref{epsest}, so, for the reason given in Remark~\ref{r:decay},
one cannot use the saddle asymptotic \eqref{EShat2} to cover these terms.
This motivates the following intermediate
estimate which allows smaller frequencies,
has explicit $\xi$ and $\freq$ dependence,
but (because of its second term) fails to be summable in $\xi$.

\begin{lem}
  For all $\freq>0$ and all $|\xi|\ge 3\freq$,
  the \FT\ of the ES kernel \eqref{ES} obeys
  \be
  |\hat\phi(\xi)| \;\le\; 9e^{-\beta}\left(
  \frac{\freq^2}{\xi^2}
  + \frac{1}{|\xi|} \right)
  ~.
  \label{EShat4}
  \ee
  \label{l:EShat4}
\end{lem}
\begin{proof}
  Bringing out the constant factor in the Fourier transform of \eqref{ES},
  we wish to bound
  $$
  e^\freq \hat\phi(\xi) = \int_{-1}^1 e^{\freq \sqrt{1-z^2} + i\xi z} dz~.
  $$
  Fixing $\beta$ and $\xi\ge 3\beta$,
  we (again) deform the contour into the upper half plane,
  break it into three pieces, and estimate each piece.
  We will judiciously choose a radius
  \be
  R = R_{\freq,\xi} = \sqrt{1+[(\xi/2\freq)^2-1]^{-1}}~,
  \qquad \mbox{ or }\quad
  (1-R^{-2})^{-1/2} = \xi/2\freq~.
  \label{R}
  \ee
  Since $(\xi/2\freq)^2 \ge 9/4$, then $1<R\le 3/\sqrt{5}$.
  Let $I_1$ be the integral along the real axis from $-1$ to $-R$,
  let $I_2$ be the integral along the semicircle $|z|=R$, $\im z \ge 0$,
  and let $I_3$ be the integral along the real axis from $R$ to $1$.
  Then $e^\freq \hat\phi(\xi) = I_1+I_2+I_3$.
  Here the branch cuts of the integrand may be taken to lie below the
  real axis.
  
  For all real $z$ the integrand has unit magnitude, giving the trivial bound
  $|I_1+I_3| \le 2(R-1)\le R^2-1$. Using \eqref{R} and $(\xi/2\freq)^2 \ge 9/4$
  we see that this is bounded
  by $(36/5)(\freq/\xi)^2 < 9(\freq/\xi)^2$, giving
  the first term in \eqref{EShat4}.

  On the upper semicircle we can limit
  the vertical exponential growth rate of $e^{\sqrt{1-z^2}}$ via
  \be
  \re \sqrt{1-z^2} \; \le \; (1-R^{-2})^{-1/2} \im z ~, \quad
  \mbox{ for all $z$ with $|z|=R$ and $\im z \ge 0$}~.
  \label{circ}
  \ee
  This is proven by setting $z=\sqrt{R^2-b^2}+ib$ and $\sqrt{1-z^2} = p+iq$,
  so that $\re (1-z^2) = p^2-q^2 = 1-R^2-2b^2$ and
  $\im (1-z^2) = 2pq = -2b\sqrt{R^2-b^2}$.
  Eliminating $q$ then solving the quadratic
  equation for $p^2$ gives
  $2p^2 = 1-R^2+2b^2 + \sqrt{(R^2-1)^2 + 4b^2}$.
  Applying the inequality $\sqrt{A^2+B^2}\le A + B^2/2A$ for $A>0$ gives
  after simplification $2p^2 \le 2b^2/(1-R^{-2})$, which
  is equivalent to \eqref{circ}.

  Applying \eqref{circ} to $I_2$, the integrand magnitude
  obeys
  $|e^{\freq \sqrt{1-z^2} + i\xi z}| \le e^{[\beta(1-R^2)^{-1/2} - \xi] \im z} =
  e^{-(\xi/2)\im z}$.
  This explains the choice \eqref{R}:
  it limits the growth rate of $e^{\freq \sqrt{1-z^2}}$ to at most half of the
  decay rate of $e^{i\xi z}$, so that decay wins.
  Now parametrizing the quarter-circle and using $\sin\theta \ge 2\theta/\pi$
  in $0\le\theta\le \pi/2$, as in the proof of Jordan's lemma, 
  $$
  |I_2| \le 2 \int_0^{\pi/2} e^{-(\xi/2) R \sin\theta} R d\theta
  \le 2R \int_0^{\pi/2} e^{-(R\xi/\pi) \theta} d\theta
  \le 2\frac{3}{\sqrt{5}}\frac{\pi}{\xi} \le \frac{9}{\xi}~,
  $$
  giving the second term in \eqref{EShat4}.
\end{proof}

\section{Phased sinc sums and proof of the main theorem}
\label{s:pf}

Firstly, to handle Fourier tails due to (exponentially small) discontinuities
at the edge of the support of $\phi$, we need the following
technical lemmas.
The first lemma bounds a conditionally convergent sinc sum,
but needs the second lemma which uniformly bounds the difference between an
exponential sum and a sinc function.
Recall the definition $\sinc x := (\sin x)/x$ for $x\neq 0$, or 1
otherwise.
  
\begin{lem}[phased sinc sum]  %
  Fix $n>0$ and $\rat>1$.
  Then there is a constant $C$
  such that for all $b\ge 1$,  $x\in\RR$, $|k|\le n/2\sigma$, and $\al>0$,
  \be
  \left|
  \sum_{|m|>b} \frac{\sin \al (mn+k)}{mn+k} e^{i(mn+k)x}
  \right|
  \; \le \; C\frac{\log b}{n}~.
  \label{phasedsinc}
  \ee
\label{l:Fourmont}
\end{lem}
\begin{proof}
  We apply Poisson summation \eqref{pois}
with grid spacing $h=2\pi/n$ to
  the top-hat function $s(x) = 1$ in $|x|\le \al$, zero otherwise.
  Since $\hat s(k) = 2\al \sinc (\al k)$, it gives
  $$
  h \!\!\sum_{|x-lh|\le \al} \hspace{-2.5ex}{\vphantom{\sum}}'  %
  \;e^{ikhl}  
  \;=\;
  2\al \sum_{m\in\ZZ} \sinc (\al (mn+k)) e^{i(mn+k)x}
  ~,\qquad k, x\in\RR, \; \al>0,
  $$
  where the prime on the sum indicates that any extremal terms
  where $|x-lh|=\al$ are to be given half their weight.
  Subtracting the $m=0$ term from both sides gives
  \be
  h \!\!\sum_{|x-lh|\le \al} \hspace{-2.5ex}{\vphantom{\sum}}'  %
  \;e^{ikhl} \;-\; 2\al\sinc (\al k) e^{ikx}
  \;=\;
  2 \al \sum_{m\neq 0}\sinc (\al (mn+k)) e^{i(mn+k)x}
  ~.
  \label{mneq0}
  \ee
  Following an idea of Fourmont \cite[Lemma~2.5.4]{fourmontthesis},
  we use a triangle inequality on \eqref{phasedsinc},
  $$
  \left|
  \sum_{|m|> b} \! \frac{\sin \al (mn+k)}{mn+k} e^{i(mn+k)x}
  \right|
   \le
  \left|
  \sum_{m\neq 0} \frac{\sin \al (mn+k)}{mn+k} e^{i(mn+k)x}
  \right|
  +
    \left|
  \sum_{1\le|m|\le b} \!\!\! \frac{\sin \al (mn+k)}{mn+k} e^{i(mn+k)x}
  \right|.
  $$
  We bound the first term by applying Lemma~\ref{l:quadr} %
  to the left-hand side of \eqref{mneq0}
  to get $C/n$, where $C$ is independent of $\al$, $x$, and $k$ in its
  allowed domain,
  and bound the second term via the harmonic sum
  $|\sum_{m=1}^b (k\pm mn)^{-1}| \le C (\log b)/n$,
  which holds since $|k|\le n/2$. The second term
  dominates. %
\end{proof}

\begin{lem}%
  Let $\rat>1$, $n>0$, and $h=2\pi/n$. Then there is a constant $C$ such that
  \be
  \biggl| h \!\! \sum_{|x-lh|\le \al} \hspace{-2.5ex}{\vphantom{\sum}}'
  \;e^{ikhl} \;-\; 2 \al \sinc (\al k) e^{ikx} \biggr|
  \; \le \; Ch
  ~, \qquad \mbox{ for all } x\in\RR, \; |k|\le \frac{n}{2\rat}, \; \al>0~.
  \label{quadr}
  \ee
  \label{l:quadr}
\end{lem}
\begin{proof}
  Note that $2 \al \sinc (\al k) e^{ikx} = \int_{x-\al}^{x+\al} e^{iky} dy =
  e^{ik(x-\al)} \sum_{j=0}^{J-1} e^{ikhj} \int_0^h e^{iky} dy + e_1h$, where
  $J$ is the integer nearest $2\al/h = n\al/\pi$,
  and $e_1$ is an end correction with $|e_1|\le 1$.
  Into this we will insert $\int_0^h e^{iky} dy = (e^{ikh}-1)/ik = he^{ikh/2}\sinc(kh/2)$.
  Note also that the sum in \eqref{quadr} is a quadrature rule with weights
  $h$ and one node per interval $[x-\al+hj,x-\al+h(j+1)]$,
  $j=0,\dots,J-1$,
  up to $\bigO(h)$ end corrections.
  Each node is offset from the left end of its interval by
  $\delta = \min_{l\in\mathbb{Z}, \; lh \ge x-\al} lh - (x-\al)$, thus
  $\sum_{|x-lh|\le \al}' e^{ikhl}
  = e^{ik(x-\al)} \sum_{j=0}^{J-1} e^{ikhj} e^{ik\delta} + e_2$,
  where $e_2$ is another end correction, $|e_2|\le 1$.
  Combining results so far,
  $$
   h \!\! \sum_{|x-lh|\le \al} \hspace{-2.5ex}{\vphantom{\sum}}'
   \;e^{ikhl} - 2 \al \sinc (\al k) e^{ikx}
   \;=\;
  h e^{ik(x-\al)} \bigl[e^{ik\delta} - e^{ikh/2} \sinc(kh/2) \bigr]
  \sum_{j=0}^{J-1} e^{ikhj} \;+\; (e_2-e_1)h~.
  $$
  The geometric sum is exactly $(1-e^{iJkh})/(1-e^{ikh}) = e^{i(J-1)kh/2} \sin (Jkh/2) / \sin (kh/2)$, so is bounded in size by $C/|kh|$, independently of $J$,
  since $|\sin(\theta/2)| \ge |\theta|/C$ for $|\theta|\le \pi/2$, with $C=\pi/\sqrt{2}$.
  But since $\delta\in[0,h)$, the factor in square brackets is
  bounded in size by $C|kh|$ for some $C$, over the domain
  $|kh|\le \pi/2$, cancelling the $1/|kh|$ blow-up. Thus all terms
  are uniformly bounded by $Ch$.
\end{proof}

\begin{rmk}[Interpretations of Lemma~\ref{l:quadr}] %
  The above lemma may be interpreted
as an error bound when applying a simple $\bigO(h)$-accurate
equispaced quadrature rule to $\int_{x-\al}^{x+\al} e^{iky} dy$.
Remarkably, oscillatory cancellation
makes the implied error constant independent of the domain width $2\al$.

A second interpretation is that it generalizes the little-known fact that
the {\em Dirichlet kernel} \cite[Sec.~11.10]{apostol}
$D_N(\theta) := \sum_{n=-N}^{N} e^{in\theta} =
\sin((N+\half)\theta) / \sin (\theta/2)$ is {\em uniformly} close
to its non-periodic analogue $2(N+\half)\sinc((N+\half)\theta)$,
throughout $N\in\mathbb{N}$ and $|\theta|\le \pi/2$ (say).
The simpler estimates needed for this proof are
$|\sinc((N+\half)\theta)|\le 1/(N+\half)|\theta|$ and
$|\sinc(\theta/2)-1|\le C|\theta|$.
\end{rmk}   %

Finally, we combine Lemma~\ref{l:Fourmont} with all of the results in Sec.~\ref{s:asymp}
to prove the main Theorem~\ref{t:ESerr}.

\begin{proof}
  Since $n$, $\rat$, and $\gamma$ are fixed, then
  $w$ and $\freq$ are proportional via \eqref{gam}.
  We wish to bound $\eps_\infty$, defined by \eqref{epsest},
  as $w\to\infty$, or, equivalently, as $\freq\to\infty$.
  Defining $\al := \pi w / n$ as the half-width of the scaled
  kernel,
  writing \eqref{epsest} in terms of the unscaled kernel \eqref{ES} gives
  \be
  \eps_\infty =
\frac{
  \max_{|k|\le N/2, \,x\in\RR} \left|
    \sum_{m\neq 0} \hat\phi(\al k+\pi w m) e^{i(k+mn)x}
  \right|
}{\min_{|k|\le N/2} |\hat\phi(\al k)|}
~,
\label{epsest2}
\ee
  We now denote $e^\freq$ times the sum in the numerator by
  \be
  G(k,x) := e^{\freq} \sum_{m\neq 0} \hat\phi(\xi_m) e^{i\xi_m x/\al}
  ~, \qquad \mbox{ where }\; \xi_m := \al k + \pi w m
  ~.
  \label{G}
  \ee
  The main task will be to prove that
  \be
  |G(k,x) | \; =\; \bigO(1)
  ~,\qquad \freq\to\infty~, \quad
  \mbox{ uniformly in $x\in\RR$, $|k|\le n/2\rat$~,}
  \label{Gunif}
  \ee
  which will imply that the numerator of \eqref{epsest2} is $\bigO(e^{-\freq})$.
  We will split the sum \eqref{G} into three ranges of $|m|$, then discard the
  phase information $e^{i\xi_mx/\al}$ in all but the tail,
  where it becomes crucial.

  We start be defining the closest range contribution by
  $$
  G_1(k,x) \;:=\;
  e^{\freq} \sum_{m\neq 0, \, |\xi_m|<3\freq} \hat\phi(\xi_m) e^{i\xi_m x/\al}~,
  $$
  which, using \eqref{gam}, involves at most five
  terms, independent of $\freq$.
  Since $\gamma<1$, each term has $|\xi_m/\beta|>1$ so is strictly
  above cutoff.
  Thus, applying the leading saddle point result \eqref{EShat2}, each term
  contributes magnitude $\bigO(1/\sqrt{\freq})$, so
  $G_1(k,x) = \bigO(1/\sqrt{\freq})$.

  The intermediate range contribution we define by
  $$
  G_2(k,x) \; :=\;
  e^{\freq} \sum_{3\freq\le|\xi_m|\le\freq^4} \hat\phi(\xi_m) e^{i\xi_m x/\al}~,
  $$
  which involves $\bigO(\freq^3)$ terms.
  Applying Lemma~\ref{l:EShat4} and using $\xi_m \sim c \freq m$
  where $c$ is a constant,
  $$
  |G_2(k,x)| \;\le\; C \freq^2 \sum_{m\le \freq^3} \frac{1}{\freq^2 m^2}
  + C \sum_{m\le \freq^3} \frac{1}{\freq m}
  = \bigO(1) + \bigO\left(\frac{\log \freq}{\freq}\right) = \bigO(1)~.
  $$

  The remaining tail contribution to \eqref{G} is
  $$
  G_3(k,x)  \;:=\;
  e^{\freq} \sum_{|\xi_m|\ge\freq^4} \hat\phi(\xi_m) e^{i\xi_m x/\al}~.
  $$
  We apply Lemma~\ref{l:EShat3}
(again noting the cancellation of $e^{\freq}$)
  and the triangle inequality to get,
  $$
  |G_3(k,x)| \; \le \; \left|\sum_{|\xi_m|\ge\freq^4} 2\frac{\sin \xi_m}{\xi_m}
   e^{i\xi_m x/\al} \right|
   + C \sum_{|\xi_m|\ge\freq^4} \frac{\freq}{\freq^{5/4}m^{5/4}}
   ~.
   $$
   For the first (conditionally convergent) sinc sum, we apply
   Lemma~\ref{l:Fourmont} with $b = \freq^3$,
   which bounds the term
   by $\bigO((\log \freq) / n\al) = \bigO((\log \freq) / \freq)$,
   uniformly over $|k|\le n/2\rat$ and $x\in\RR$.
   The second term is summable so is $\bigO(\freq^{-1/4})$.
   Thus $|G_3(k,x)| = \bigO(\freq^{-1/4})$.
   Since $G = G_1 + G_2+G_3$, \eqref{Gunif} is proved.
   
   The only remaining task is a
   lower bound on the denominator in \eqref{epsest2}.
   We exploit Theorem~\ref{t:EShat} below cutoff, i.e.\ \eqref{EShat1}.
   The minimum occurs at the edge of the usable band, $|k|=N/2 = n/2\rat$,
   i.e.\ $|\xi| := \al n/2\rat = \pi w/2\rat$,
   i.e.\ scaled frequency $\rho_e := \pi w/2\rat\freq = (\gamma(2\rat-1))^{-1}$
   using \eqref{gam}.
   Then the below-cutoff asymptotic \eqref{EShat1} implies an upper bound on the inverse of the denominator,
   $$
   \bigl|\hat\phi(\pi w/2\rat)\bigr|^{-1} =
   \bigO\bigl(\sqrt{\freq} e^\freq e^{-\freq\sqrt{1-\rho_e^2}}\bigr)
   =
   \bigO\bigl(\sqrt{\freq} e^\freq e^{-\pi w \sqrt{\gamma^2(1-1/\rat) - (1-\gamma^2)/4\rat^2}}\bigr)
   ~,
   $$
   after simplification.
   Inserting this as the denominator of \eqref{epsest2},
   and recalling \eqref{G}--\eqref{Gunif}, the factor $e^\freq$ is cancelled.
   Since $w$ is proportional to $\freq$, this proves \eqref{ESerr}.
\end{proof}

\section{Connections between optimal rate spreading kernels}
\label{s:optim}

Here we link the ES, KB, and PSWF kernels asymptotically,
discuss their common
exponential convergence rate, conjecture that it is optimal,
and conclude with a question.

The first connection starts with the KB kernel
and inserts the large-argument asymptotic
$I_0(y) \sim e^y/\sqrt{2\pi y}$ \cite[(10.3.4)]{dlmf},
to give, for any fixed $0<a<1$,
\be
\phi_{KB,\freq}(z) \;\sim\; \frac{e^{\freq(\sqrt{1-z^2}-1)}}{(1-z^2)^{1/4}}~,
\qquad z\in[-a,a]~, \qquad \mbox{ uniformly as }\freq\to\infty~,
\label{KBasym}
\ee
which is the ES kernel \eqref{ES} with an extra algebraic prefactor.
Experimentally dropping this prefactor in fact led us to the ES kernel
in \cite{finufft}.
The closeness of \eqref{KBasym} to KB, and the persistent algebraic
difference between each of them and ES, is shown by Fig.~\ref{f:kernel}(b,e).

In turn, KB is connected to the PSWF of order zero, $\psi_0$.
While it is frequently stated in signal processing literature that KB is a ``good'' approximation
to the PSWF \cite{kaiser,jackson91} \cite[Sec.~7.5.3]{DTSP},
we cannot find any quantification of how close in a mathematical sense
(even in, say, \cite{osipov,dunster86,dunster17}).
One definition \cite{SlepianI,osipov} of $\psi_0$
is the function with support in $[-1,1]$ with minimal $L^2$-norm (energy)
outside the frequency interval $[-\freq,\freq]$.
(Note that the PSWF parameter, usually called $c$, is set at $c=\freq$.)
Slepian \cite[(1.4)]{slepian65} derived
the large-$\freq$ asymptotics,
\be
\psi_0(z) = \left\{\begin{array}{ll}C \frac{e^{\freq\sqrt{1-z^2}}}{(1-z^2)^{1/4}}
(1+\sqrt{1-z^2})^{-1/2}\bigl(1 + \bigO(\freq^{-1})\bigr)
~,&
\freq^{-1/2} \le |z| \le 1-\freq^{-1}\\
C I_o(\freq \sqrt{1-z^2}) \bigl(1 + \bigO(\freq^{-1})\bigr)
~,&
1-\freq^{-1} \le |z| \le 1~.
\end{array}\right.
\label{Slep}
\ee
Thus for all $|z|\ge \freq^{-1/2}$, normalizing $\psi_0(0)=1$,
this matches the KB kernel asymptotic, apart from a factor
$$\sqrt{2}\,(1+\sqrt{1-z^2})^{-1/2}$$
whose range is only $[1,\sqrt{2}]$.
This asymptotic ratio function between PSWF and KB
we have not found in either signal processing or mathematics literature,
and it helps explain heuristically the similar performance of the kernels.
The inverse of their ratio is
plotted in Fig.~\ref{f:kernel}(b,e) (green dashed line).
Curiously, we find empirically
that the {\em hybrid} form (Fig.~\ref{f:kernel}(b,e), red dotted line)
\be
\psi_\tbox{SlepH}(z) =  CI_o(\freq \sqrt{1-z^2})(1+\sqrt{1-z^2})^{-1/2}
\label{SlepH}
\ee
is a much better approximation to $\psi_0$ in $[-1,1]$ than
\eqref{Slep}, with relative error uniformly $<0.2/\freq$.

Other asymptotics for the PSWF are known.
Inside the central (``turning point'') region
$|z| = \bigO(\freq^{-1/2})$,
the PSWF tends to the Gaussian
$\psi_0(z) = C e^{-\beta z^2 /2} + \bigO(\freq^{-1})$.
which has a width differing by only $4\%$ from that of the
optimal truncated Gaussian (shown in \cite[Fig.~1.1(a)]{finufft}).
However, this is less useful than the above forms because
of the inferior convergence rate of the
Gaussian discussed in the introduction.
This Gaussian limit can be derived by expansion in Hermite functions
\cite[\S 3.25]{meixner} %
\cite[Sec.~8.6]{osipov}.
Finally, WKBJ large-$\freq$ asymptotics of the PSWF
due to Dunster \cite{dunster86},
Ogilvie \cite[Sec.~4.4]{ogilvie}, and others,
uniformly cover more of $[-1,1]$, but involve changes of variable
that obscure any connection to KB or ES.

We now compare the aliasing error convergence rates of the three kernels,
which is most relevant in applications.
Remark~\ref{r:fourmont} stated that in the limit $\gamma\to1^{-}$
the rate of the ES kernel matches that of KB.
This rate $e^{-\pi w\sqrt{1-1/\rat}}$ may be traced to
the optimal $\freq$ choice \eqref{gam},
to both kernels
having Fourier tails of order $e^{-\freq}$ times their values at zero frequency
(see \eqref{Gunif} and Fig.~\ref{f:kernel}(f)),
and to the exponential-of-semicircle
form of $\hat\psi$ in the output band $|k|\le N/2$
(see \cite[Prop.~2]{fourmont} and \cite[Fig.~3.1(b)]{finufft}).
Does this rate also hold for the PSWF?
Since $\psi_0$ is a normalized eigenfunction of the projection operator
($Q_c$ in \cite{osipov}) onto the frequency band $[-\freq,\freq]$,
its eigenvalue $\mu_0\le 1$ gives the squared mass
$2\pi \int_{|\xi|<\freq}|\hat\psi_0(\xi)|^2 d\xi$,
and the remaining mass is
$1-\mu_0 =2\pi \int_{|\xi|>\freq}|\hat\psi_0(\xi)|^2 d\xi$.
It was proven by Fuchs \cite{fuchs} that
\be
1 - \mu_0 \;\sim\; 4\sqrt{\pi\freq} e^{-2\freq} ~, \qquad \freq\to\infty~,
\label{fuchs}
\ee
which shows that the $L^2$-norm of $\hat\psi_0$ outside of
$[-\freq,\freq]$ is exponentially small with rate $e^{-\freq}$,
which is indeed the same rate as in the ES and KB error bounds.
Yet the $L^2$-norm does not bound the %
sum appearing in \eqref{epsest}, so a
bound on $\eps_\infty$ for the PSWF remains, to the author's knowledge,
heuristic.
However, it strongly
suggests the following.
\begin{cnj}   %
  Fix $N$ and the upsampling factor $\rat>1$, with $n=\rat N$. Then
  $c_{\tbox{\rm optim},\rat} := \pi \sqrt{1-1/\rat}$
is the supremum of all values $c$ for which there
exists a family of kernels $\psi_w$
of support $[-\pi w/n,\pi w/n]$ with aliasing error bound \eqref{epsest}
obeying
\be
\eps_\infty = \bigO(e^{-cw})~, \qquad \mbox{ as } w\to\infty~.
\ee
\end{cnj} %
This is consistent with the fact that numerical kernel optimization
has produced only minimal reduction in errors relative to known kernels \cite{jackson91,fessler,l2jacob}
(unless special assumptions on the power spectrum $f_k$ are made \cite{nestler}).
The fact that the three kernels share this optimal rate
tells us that their differing algebraic factors are
surprisingly irrelevant for frequency localization.

\begin{rmk}[Fast look-up tables for any smooth kernel] %
  The ease of numerical evaluation of a kernel is often claimed to be
  decisive in its choice:
  this was used to justify the KB over the PSWF kernel
  \cite{kaiser,kaiserinterview,jackson91,fourmont,l2jacob},
  and %
  the ES over the KB \cite{finufft}.
  However, look-up tables of coefficients of piecewise high-order
  polynomial interpolants
  evaluated via Horner's rule, and modern open-source compiler vectorization,
  as used in FINUFFT \cite[Sec.~5.3]{finufft}, mean that
  any smooth kernel with optimal rate, including the PSWF, could currently be
  used efficiently without loss of accuracy.
  Yet, it is possible that future hardware will have
  relatively faster ${\tt exp}$ evaluations, again making the
  simplest such kernel, ES, preferable.
\end{rmk}    %

We finish with an open-ended question:
ignoring algebraic prefactors,
do all kernels on $[-1,1]$ that share the optimal rate $c_{\tbox{\rm optim},\rat}$
have the asymptotic exponential form
$e^{\freq\sqrt{1-z^2}}$, and, if so, why?

\section*{Acknowledgments}

The author has benefited from discussions with Charlie Epstein, Jeremy Magland, Leslie Greengard, Ludvig af Klinteberg, Mark Dunster, and Daniel Potts,
and thanks Roy Lederman for MATLAB codes for accurate
PSWF evaluation.
The Flatiron Institute is a division of the Simons Foundation.

\bibliographystyle{abbrv}
\bibliography{alex}
\end{document}